\documentclass{article}
\usepackage[a4paper, left=3cm, right=3cm, top=3cm]{geometry}
\usepackage{amssymb}
\usepackage{amsmath}
\usepackage{amsthm}
\usepackage{graphicx}

\title{Connectedness of random set attractors}
\author{
	Michael Scheutzow\thanks{
	Institut f\" ur Mathematik, MA 7-5, Fakult\" at II, Technische Universit\" at Berlin, 
	Stra\ss e des 17. Juni 136, \newline 10623 Berlin, Germany; 
	e-mail: \texttt{ms@math.tu-berlin.de}, \texttt{vorkastn@math.tu-berlin.de}  } 
	\and
	Isabell Vorkastner\footnotemark[1] }
\date{\today}

\begin{document}

\maketitle

\theoremstyle{plain}
\newtheorem{theorem}{Theorem}[section]
\newtheorem{lemma}[theorem]{Lemma}
\newtheorem{proposition}[theorem]{Proposition}
\newtheorem{corollary}[theorem]{Corollary}
\theoremstyle{definition}
\newtheorem{definition}[theorem]{Definition}
\newtheorem{example}[theorem]{Example}
\newtheorem{remark}[theorem]{Remark}

\abstract{
	We examine the question whether random set attractors for continuous-time random dynamical systems 
	on a connected state space are connected. In the deterministic case, these attractors are known to
	be connected. In the probabilistic setup, however, connectedness has only been shown under stronger
	connectedness assumptions on the state space. 
	Under a weak continuity condition on the random dynamical system  
	we prove connectedness of the pullback attractor on a connected space. 
	Additionally, we provide an example of a weak  random set attractor of a random dynamical system with even 
        more restrictive continuity assumptions on an even path-connected space
	which even  attracts all bounded sets and which is not connected. On the way to proving connectedness of a pullback attractor 
        we prove a lemma which may be of independent interest and which holds without the assumption that the state space is connected. 
        It states that even though pullback convergence to the attractor allows for exceptional nullsets which may depend on the compact set, 
        these nullsets can be chosen independently of the compact set (which is clear for $\sigma$-compact spaces but not at all clear for 
        spaces which are not $\sigma$-compact).}

\section{Introduction}

While attractors for (deterministic) dynamical systems have been studied for a long time, attractors for {\em random dynamical systems} were 
only introduced and studied in the nineties of the last century. The question of connectedness of a random pullback attractor was first 
addressed in the seminal paper \cite{Crauel1994}. Proposition 3.13 of that paper states that if a random dynamical system in discrete or 
continuous time taking values in a connected Polish space admits a pullback attractor $A$ (in the sense that $A$ attracts every bounded set in the pullback sense almost surely) then  $A$ is almost surely connected. Later, a gap was found in the proof of that proposition and an example in  
\cite{Gobbino1997} shows that the claim does not even hold  true in the deterministic case when time is discrete.
Positive results (in discrete and continuous time) have been found in \cite{Crauel2003} under the additional condition that  any compact
set in the state space can be covered by a connected compact set (a property which clearly does not hold in the example in \cite{Gobbino1997}).
\\
The aim of this paper is to examine the question whether random set attractors of continuous-time random dynamical systems on a 
connected state space are connected. 
\\
In this paper, we distinguish between two kinds of random set attractors, pullback and weak attractors (precise definitions will be provided in the next section). By  {\em set} attractor we mean 
an attractor which either attracts every deterministic compact set or every deterministic bounded set (we will state explicitly in each case if we want the attractor to attract every compact or even every bounded set). Pullback and weak attractors 
differ in the type of convergence of compact (or bounded) sets under the action of the random dynamical system to the attractor. {\em Pullback} stands for almost sure convergence and {\em weak} for convergence in probability.
Both of these set attractors are known to be (almost surely) unique, see \cite[Lemma 1.3]{Gess2016}. 
\\
In Section \ref{pullback_section}, we consider pullback attractors for continuous-time random dynamical systems taking values in a 
connected Polish space. 
Under a rather weak continuity assumption on the random dynamical system which we call {\em pullback continuity}
we show that the pullback attractor (if it exists) is almost surely connected (even if it is only required to attract all compact sets). 
The first lemma in that section may be of independent interest.  It states that even though pullback convergence to the attractor allows for exceptional nullsets which may depend on the compact set, 
these nullsets can be chosen independently of the compact set (even if the space is not $\sigma$-compact). This lemma does not assume the state space to be connected. The result allows us to argue pathwise (for fixed $\omega$) in the proof of the main result.
\\
In Section \ref{weak_attractor_section} we provide an example of a random dynamical system on a path-connected state space where the weak
attractor is not connected. In that example the random dynamical system enjoys even stronger continuity properties than in the previous section and the attractor even attracts all bounded and not just compact sets. The state space in that example is the same as that in \cite{Gobbino1997} but the random dynamical system on that space is more sophisticated.
\\
Apart from set attractors for continuous-time system other types of random attractors such as {\em random  point attractors} or {\em random Hausdorff-Delta-attractors} have been studied in the literature 
either in the pullback or weak sense (\cite{Crauel2003}, \cite{SMWB2017}). These are generally not connected even if the ambient space is connected and the attractors are chosen to be {\em minimal} 
(unlike set attractors they are generally not unique).  
As an example for a disconnected minimal  point attractor consider the scalar differential
equation $\mathrm{d} x = (x - x^3) \, \mathrm{d} t$ on the interval $[0,1]$. Each trajectory converges 
to $\left\{0 \right\} $ or $ \left\{ 1 \right\}$. Hence, $\left\{0 \right\} \cup \left\{ 1 \right\}$ is
the minimal (pullback or weak) point attractor (while the set attractor is the whole interval $[0,1]$).

\section{Notation and preliminaries}
Let $(X,d)$ be a Polish (i.e.~separable complete metric) space with Borel $\sigma$-algebra $\mathcal{B}(X)$
and $(\Omega, \mathcal{F}, \mathbb{P}, \theta)$ be a metric dynamical system, i.e.
$(\Omega, \mathcal{F},\mathbb{P})$ is a probability space and $(\theta_t)_{t\in \mathbb{R}}$ a group of
jointly measurable maps on $(\Omega, \mathcal{F},\mathbb{P})$ such that $\theta_0=\mathrm{id}$ with invariant measure $\mathbb{P}$.
Denote by $\bar{\mathcal{F}}$ the completion of $\mathcal{F}$ with respect to $\mathbb{P}$.
We further denote by $\bar{\mathbb{P}}$ the (unique) extension of $\mathbb{P}$ to $\bar{\mathcal{F}}$. 
\\
Let $\varphi : \mathbb{R}_+ \times \Omega \times X \rightarrow X$ be jointly measurable, $\varphi_0 (\omega,x) =x$, 
$\varphi_{s+t} (\omega,x) = \varphi_t (\theta_s \omega, \varphi_s (\omega,x))$ for all $x \in X$,
and  $ x \mapsto \varphi_t (\omega, x)$ continuous,
$s,t \in \mathbb{R}_+$ and $\omega \in \Omega$. 
Then, $\varphi$ is called a \emph{cocycle} and the collection 
$(\Omega, \mathcal{F}, \mathbb{P}, \theta,\varphi)$ is called a \emph{random dynamical system (RDS)},
see \cite{Arnold1988} for a comprehensive treatment. We call an RDS \emph{pullback continuous} if 
$t \mapsto \varphi_{t}(\theta_{-t}\omega,x)$ is continuous for each $\omega \in \Omega$ and $x \in X$.
\\
A \emph{semi-flow} $\phi:\left\{ - \infty < s \leq t < \infty \right\} \times \Omega \times X \rightarrow X$ 
satisfies $\phi_{s,u}(\omega,x) = \phi_{t,u}(\omega,\cdot) \circ \phi_{s,t}(\omega,x)$,
$\phi_{s,t}(\omega,x) = \phi_{s+h,t+h}(\theta_h \omega,x)$ and $\phi_{s,s}(\omega,x) =x$ for $\omega \in \Omega$,
$x \in X$, $ h \in \mathbb{R}$ and $- \infty <s \leq t \leq u <\infty$.
There is a one-to-one relation between cocycles and semi-flows. 
One can either define a semi-flow by $\phi_{s,t}(\omega,x):= \varphi_{t-s}(\theta_{s}\omega,x)$ or 
a cocycle by $\varphi_t(\omega,x) := \phi_{0,t}(\omega,x)$.
We say a semi-flow respectively RDS is \emph{jointly continuous} if $(s,t,x) \mapsto \phi_{s,t} (\omega,x)$ respectively $(s,t,x) \mapsto \varphi_{t-s}(\theta_{s}\omega,x)$ is continuous.  
Note that a jointly continuous RDS is pullback continuous but the converse does not necessarily hold true. 
\\
For a set $A \subset X$ we denote
\begin{align*}
	A^\varepsilon := \left\{ x \in X : d (x ,A): = \inf_{a \in A} d(x,a) < \varepsilon \right\}.
\end{align*}
\begin{definition}
	A family $\left\{ A (\omega ) \right\}_{\omega \in \Omega}$ of non-empty subsets of $X$
	is called 
	\begin{enumerate}
	\item[(i)] 
		a \emph{random compact set} if it is $\mathbb{P}$-almost surely a compact set and 
		$\omega \mapsto d(x,A(\omega))$ is $\mathcal{F}$-measurable 
		for each $x \in X$.
	\item[(ii)]
		\emph{$\varphi$-invariant} if $\varphi_t (\omega, A(\omega)) = A(\theta_t \omega )$ 
		for almost all $\omega \in \Omega$,\, $t \in \mathbb{R}_+$.
	\end{enumerate}
\end{definition}
\begin{definition}
	Let $(\Omega, \mathcal{F}, \mathbb{P}, \theta, \varphi)$ be a random dynamical system. 
	A random compact set $A$ is called a \emph{pullback attractor} if 
	it satisfies the following properties
	\begin{enumerate}
	\item [(i)] $A$ is $\varphi$-invariant
	\item [(ii)] for every compact set $B \subset X$ 
		\begin{align*}
			\lim_{t \rightarrow \infty} \sup_{x \in B} 
				d( \varphi_t (\theta_{-t}\omega ,x), A(\omega)) =0 
				\qquad  \mathbb{P}\textrm{-almost surely}.
		\end{align*}
	\end{enumerate}
If the convergence in (ii) is merely in probability, then $A$ is called a  \emph{weak attractor}.
\end{definition}

\section{Pullback attractor}
\label{pullback_section}

In this section, we show that the pullback attractor of a pullback continuous RDS 
on a connected space is connected.
The pullback attractor attracts any compact set almost surely. We prove that the nullsets where it may not converge can be be chosen independently of the compact set. This allows us to analyze the RDS pathwise 
and to use similar arguments as in the deterministic proof of \cite[Theorem 3.1]{Gobbino1997}.

\begin{lemma}
	\label{omega-wise_convergence}
	Let $A$ be the pullback attractor of the pullback continuous RDS $\varphi$.
	Then, there exists some $\hat{\Omega} \in \mathcal{F}$ with $\mathbb{P}(\hat{\Omega})=1$
	such that for any $\omega \in \hat{\Omega}$ and compact set $K \subset X$,
	\begin{align*}
		\lim_{t \rightarrow \infty} \sup_{x \in K} d(\varphi_t ( \theta_{-t} \omega, x), A(\omega)) = 0.
	\end{align*}
\end{lemma}

\begin{proof}
	First, we consider convergent sequences in $X$. Let
	\begin{align*}
		\hat{c} := \left\{ ( x_\infty , x_1 , x_2, x_3, \dots ) \in X^\mathbb{N} : 
		 d(x_n,x_\infty) \leq \frac{1}{n} \textrm{ for all } n \in \mathbb{N} \right\}
	\end{align*}
	which is closed in the Polish space $X^\mathbb{N}$ and hence itself a Polish space. Further, let
	\begin{align*}
		M (\omega ) := \bigcup_{n \in \mathbb{N}} 
		\bigcap_{m \in \mathbb{N}} \; \bigcup_{q \in \mathbb{Q}, q \geq m}
		\; \bigcup_{k \in \mathbb{N} \cup \left\{\infty \right\} }
		\left\{ (x_\infty , x_1, x_2, \dots) \in \hat{c} : 
		\varphi_q (\theta_{-q} \omega, x_k) \in A(\omega) ^\frac{1}{n} \right\} ^c
	\end{align*}
	be the set of sequences of $\hat{c}$ that are not uniformly attracted.
	By measurability of $\varphi$ and $A$, the graph of $M$ is measurable.\\
	Assume there is a subset $\tilde{\Omega} \in \mathcal{F}$ with $\mathbb{P}( \tilde{\Omega} ) > 0$
	such that $M (\omega)  \neq \emptyset$ for all $\omega \in \tilde{\Omega}$.
	Define
	\begin{align*}
		\tilde{M} (\omega) := \begin{cases} 
			M(\omega ) & \textrm{if } \omega \in \tilde{\Omega} \\
			\hat{c} & \textrm{else}.	\end{cases}		
	\end{align*}
	Then the graph of $M$ is in $\mathcal{F} \times \mathcal{B}(X)$ and hence in 
	$\bar{\mathcal{F}} \times \mathcal{B}(X)$.
	Note that $\bar{\mathcal{F}}$ is closed under the Souslin operation 
	(see \cite[Example 3.5.20 and Theorem 3.5.22]{Srivastava1998}). 
	Hence, \cite[Corollary of Theorem 7]{Leese1974} 
	(see also the survey by Wagner \cite[Theorem 3.4]{Wagner1980}) implies the existence of 
	a $\bar{\mathcal{F}}$-measurable selection 
	$x (\omega) =(x_\infty (\omega), x_1 (\omega), x_2 (\omega),\dots) \in \tilde{M} (\omega)$.
	The set $\bigcup_{k \in \mathbb{N} \cup \left\{\infty \right\} }  \left\{ x_k (\omega) \right\}$
	is sequentially compact for each $\omega \in \Omega$.
	By the same arguments as in \cite[Proposition 2.15]{Crauel2003},
	there exists some deterministic compact set $\tilde{K} \subset X$ such that 
	\begin{align*}
		\bar{\mathbb{P}} \left( x_k (\omega) \in \tilde{K} \textrm{ for all } k 
		\in \mathbb{N} \cup \left\{ \infty \right\} \right) 
		> 1 - \mathbb{P}( \tilde{\Omega} ) .
	\end{align*}
	Using the definition of $\tilde{\Omega}$ and $\hat{M}$ it follows that
	\begin{align*}
		\bar{\mathbb{P}} \left( x(\omega) \in M (\omega) \textrm{ and } x_k (\omega) \in \tilde{K} 
		\textrm{ for all } k 
		\in \mathbb{N} \cup \left\{ \infty \right\} \right) 
		>0 .
	\end{align*}
	This contradicts the fact that the pullback attractor attracts $\tilde{K}$ almost surely.
	Hence, $M(\omega) = \emptyset$ almost surely. 
	Using pullback continuity of $\varphi$, it follows that
	there exists some $\hat{\Omega} \in \mathcal{F}$ with $\mathbb{P}(\hat{\Omega})=1$
	such that for any $\omega \in \hat{\Omega}$ and $(x_\infty , x_1, x_2, \dots) \in \hat{c}$,
	\begin{align}
	\label{convergent_sequence}
		\lim_{t \rightarrow \infty} \sup_{k \in \mathbb{N} \cup \left\{ \infty \right\}} 
		d(\varphi_t ( \theta_{-t} \omega, x_k), A(\omega)) = 0.
	\end{align}
	Now, assume there exists some compact set $K$, $\varepsilon >0$, 
	$\omega \in \hat{\Omega}$ and sequence $t_m$ going to infinity such that 
	$ \varphi_{t_m} (\theta_{-t_m} \omega, K) \not\subset A(\omega)^\varepsilon$ for all $m\in \mathbb{N}$.
	Hence, there are $y_m \in K$ such that 
	$ \varphi_{t_m} (\theta_{-t_m} \omega, y_m) \not\in A(\omega)^\varepsilon$ for all $m\in \mathbb{N}$.
	Since $K$ is compact, there is a convergent subsequence $y_{m_k}$ with 
	$ y_\infty := \lim_{k \rightarrow \infty} y_{m_k}$
	and $(y_\infty, y_{m_1}, y_{m_2}, \dots) \in \hat{c}$ which is a contradiction to \eqref{convergent_sequence}.
\end{proof}

\begin{remark}
	The statement of Lemma \ref{omega-wise_convergence} remains true for 
	pullback attractors of RDS in discrete time.
\end{remark}

\begin{lemma}
	\label{step3_pullback_attractor}
	Let $A$ be the pullback attractor of the RDS $\varphi$. For $\delta >0$ 
	there exist compact sets $K_n \subset X$ and $t_n \geq 0$, $n \in \mathbb{N}$ such that
	\begin{align*}
		\mathbb{P} \left( \varphi_{t_n} \left( \theta_{-t_n} \omega, K_n \right) \supset A( \omega)
		\textrm{ and } \varphi_t \left( \theta_{-t} \omega, K_n \right) \subset A (\omega)^\frac{1}{n} 
		\textrm{ for all } t\geq t_n ,  n \in \mathbb{N} \right) \geq 1- \delta.
	\end{align*}
\end{lemma}

\begin{proof}
	Let $n \in \mathbb{N}$. 
	By \cite[Proposition 2.15]{Crauel2003} there exists some compact set $K_n \subset X$ such that
	\begin{align}
		\label{compact_set}
		\mathbb{P} \left( A (\omega) \subset K_n \right) \geq 1- \frac{\delta}{2^{n+1}}.
	\end{align}
	The definition of the pullback attractor implies that there exists some $t_n >0$ such that
	\begin{align}
		\label{time_t^ast}
		\mathbb{P} \left( \varphi_t \left( \theta_{-t} \omega, K_n \right) \subset A (\omega)^\frac{1}{n}
		\textrm{ for all } t \geq t_n \right) \geq 1 - \frac{\delta}{2^{n+1}}.
	\end{align}
	By $\varphi$-invariance of $A$, $\theta$-invariance of $\mathbb{P}$ and \eqref{compact_set} it follows that
	\begin{align*}
		\mathbb{P} \left( \varphi_{t_n} \left( \theta_{- t_n} \omega, K_n \right) \supset A(\omega) \right)
		\geq 1 - \frac{\delta}{2^{n+1}}.
	\end{align*}
	Combining this estimate and \eqref{time_t^ast}, we conclude
	\begin{align*}
		\mathbb{P} \left( \varphi_{t_n} \left( \theta_{-t_n} \omega, K_n \right) \supset A( \omega)
		\textrm{ and } \varphi_t \left( \theta_{-t} \omega, K_n \right) \subset A (\omega)^\frac{1}{n} 
		\textrm{ for all } t\geq t_n \right) \geq 1- \frac{\delta}{2^n}
	\end{align*}
	which implies the claim.
\end{proof}

\begin{theorem}
	\label{pullback_attractor_connected}
	Let $X$ be a connected Polish space and $\varphi$ be a pullback continuous RDS. If there exists a pullback 
	attractor $A$, then $A$ is almost surely connected.
\end{theorem}

\begin{proof}
	Assume $A$ is not connected with positive probability. 
	By Lemma \ref{omega-wise_convergence} and \ref{step3_pullback_attractor} 
	we can choose $\tilde{\Omega} \in \mathcal{F}$ 
	with $\mathbb{P}(\tilde{\Omega}) >0$, compact sets $K_n \subset X$ and a sequence $t_n$ 
	such that for any $\omega \in \tilde{\Omega}$, $n\in \mathbb{N}$ and compact set $K \subset X$ it holds that
	\begin{itemize}
		\item $A(\omega)$ is not connected, 
		\item $ \lim_{t \rightarrow \infty} \sup_{x \in K} d(\varphi_t ( \theta_{-t} \omega, x), A(\omega)) = 0$,
		\item $\varphi_{t_n} \left( \theta_{-t_n} \omega, K_n \right) \supset A( \omega)$ and
			$\varphi_t \left( \theta_{-t} \omega, K_n \right) \subset A (\omega)^\frac{1}{n}$ for all $t\geq t_n$.
	\end{itemize}
	Fix $\omega \in \tilde{\Omega}$. For this fixed $\omega$ we will follow the idea of the proof in the
	deterministic case (see \cite[Theorem 3.1]{Gobbino1997}). Note however that Step 3 below requires some extra argument in our case. \\
	\emph{Step 1:}
	Let $A(\omega) = A_1  \cup A_2$, where $A_1 $
	and $A_2 $ are nonempty, disjoint, compact sets. There exists some $\varepsilon > 0$ such that
	$A_1 ^{\varepsilon} \cap A_2 ^{\varepsilon} = \emptyset$. Define
	\begin{align*}
		X_1 := \left\{ x \in X : \textrm{there exists some } t \textrm{ such that }
			\varphi_s ( \theta_{-s} \omega,x) \in A_1 ^{\varepsilon }
			\textrm{ for all } s \geq t \right\} \; \, \\
		X_2 := \left\{ x \in X : \textrm{there exists some } t \textrm{ such that }
			\varphi_s ( \theta_{-s} \omega,x) \in A_2 ^{\varepsilon }
			\textrm{ for all } s \geq t \right\}.
	\end{align*}
	If we show that $X_1$ and $X_2 $ are disjoint nonempty open sets with 
	$X_1 \cup X_2 = X$, then we found a contradiction to $X$ being connected.
	Obviously, $X_1 \cap X_2 = \emptyset$. \\
	\emph{Step 2:}
	We show that $X_1 \cup X_2 = X$.\\
	Let $x \in X$. By definition of $\tilde{\Omega}$, there exists some $t > 0$ such that
	$\varphi_s (\theta_{-s} \omega,x) \in A(\omega) ^{\varepsilon }$ for all $s \geq t $.
	Define
	\begin{align*}
		S_{t} := \left\{ \varphi_s (\theta_{-s} \omega,x) : s \geq t \right\}.
	\end{align*}
	Then, $S_{t}  \subset A(\omega)^{\varepsilon }$ and $S_{t} $ is connected
	by pullback continuity. Therefore, $S_{t} $ is either totally contained in 
	$A_1^{\varepsilon }$ or totally contained in $A_2^{\varepsilon }$.\\
	\emph{Step 3:}
	We show that $X_i \neq \emptyset$ for $i = 1,2$. \\
	Let $n \in \mathbb{N}$ with $\frac{1}{n} \leq \varepsilon$.
	By definition of $\tilde{\Omega}$, 
	$\varphi_{t_n} \left( \theta_{-t_n} \omega, K_n \right) \supset A( \omega)$ and
	$\varphi_t \left( \theta_{-t} \omega, K_n \right) \subset A (\omega)^\varepsilon$ for all $t\geq t_n$
	for some $n \in \mathbb{N}$.
	Hence, there exists $x  \in K_n \subset X$ such that 
	$\varphi_{t_n} \left( \theta_{-t_n} \omega, x \right) \in A_i$. By continuity 
	in time, $\varphi_{t} \left( \theta_{-t} \omega, x \right) \in A_i^\varepsilon $
	for all $t \geq t_n$. \\
	\emph{Step 4:}
	We show that $X_i$ is open for $i=1,2$. \\
	Assume that $X_i$ is not open. Then, there exists an $x \in X_i$, a sequence $x_k$ converging to $x$ 
	and a sequence $s_k$ converging to infinity such that 
	$\varphi_{s_k} \left( \theta_{-s_k} \omega, x_k \right) \notin A_i ^\varepsilon$ for all $k \in \mathbb{N}$.
	By definition of $\tilde{\Omega}$, there exists some $s > 0$ such that
	$\varphi_t (\theta_{-t} \omega,x_k) \in A(\omega) ^{\varepsilon }$ for all $k \in \mathbb{N}$ and $t \geq s $. 
	Since $x \in X_i$, $x_k$ is converging to $x$ 
	and $\varphi$ is continuous in the state space, there exists some $k^\ast$ such that
	$\varphi_s (\theta_{-s} \omega,x_k) \in A_i ^{\varepsilon }$ for $k \geq k^\ast$. 
	Using pullback continuity, it follows that
	$\varphi_t (\theta_{-t} \omega,x_k) \in A_i ^{\varepsilon }$ for $t \geq s$ and $ k \geq k^\ast$ which is a 
	contradiction to the definition of $x_k$.
\end{proof}

\section{Weak attractor}
\label{weak_attractor_section}

The question arises whether the result in the previous section can be extended to weak attractors. 
In contrast to pullback attractors, convergence to weak attracors is merely in probabilty.
We give an example of an RDS where the weak attractor is not connected.
In addition to the assumption on the RDS and state space of Section \ref{pullback_section}, this example has a
jointly continuous RDS, a path-connected state space and every bounded set converges to the
attractor.

\begin{example}
	\emph{Step 1: The metric space.}
	We choose the same metric space as in \cite[Remark 5.2]{Gobbino1997}.
	Set $s_n = \sum_{i=0}^{n} 2^{-i}$ for $n \in \mathbb{N}_0$. 
	Let us consider the following sets in $\mathbb{R}^2$:
	\begin{align*}
		P_{-\infty} &:= (-1,0), \quad   P_{\infty} := (2,0),\\
		P_n &:= (s_{n-1},0), \quad P_{-n} := (1-s_n,0), \\
		X_n^L &:= \left\{ (x,y) \in \mathbb{R}^2 :  x = s_{n-1} + \lambda \, 2^{-n-1} \textrm{ and } 
			y= \lambda \, 2^{-n}  \textrm{ for some } \lambda \in [0,1] \right\}, \\
		X_n^R &:= \left\{ (x,y) \in \mathbb{R}^2 :  x = s_{n-1} + (2- \lambda) \, 2^{-n-1} \textrm{ and } 
			y= \lambda \, 2^{-n}  \textrm{ for some } \lambda \in [0,1] \right\} , \\
		X_{-n}^L &:= \left\{ (x,y) \in \mathbb{R}^2 :  x = 1-s_n + \lambda \, 2^{-n-1} \textrm{ and } 
			y= \lambda \, 2^{n}  \textrm{ for some } \lambda \in [0,1] \right\}, \\
		X_{-n}^R &:= \left\{ (x,y) \in \mathbb{R}^2 :  x = 1- s_n + (2- \lambda) \, 2^{-n-1} \textrm{ and } 
			y= \lambda \, 2^{n}  \textrm{ for some } \lambda \in [0,1] \right\} , \\
		X_{-\infty} &:= \left\{ (-1,y) \in \mathbb{R}^2 : y \geq 0 \right\}, \\
		Y &:= \left\{ (x,y) \in \mathbb{R}^2 : y \leq 0, \left( x - 0.5 \right)^2 +y^2 = 2.25 \right\}
	\end{align*}
	and
	\begin{align*}
		X_z := X_z^L \cup X_z^R
	\end{align*}
	for $n \in \mathbb{N}_0$ and $z \in \mathbb{Z}$. 
	The sets $X_z$ are the two equal sides of isosceles triangles in the halfplane 
	with base $P_z P_{z+1}$ and height $2^{-z}$. The left- respectively right-hand
	side of $X_z$ is denoted by $X_z^L$ respectively $X_z^R$.
	Finally we define the complete metric space
	\begin{align*}
		X:= \bigcup_{z \in \mathbb{Z}}^\infty X_n \cup X_{-\infty} \cup Y
	\end{align*}
	\begin{figure}[ht]
		\center
		\includegraphics[height=6cm]{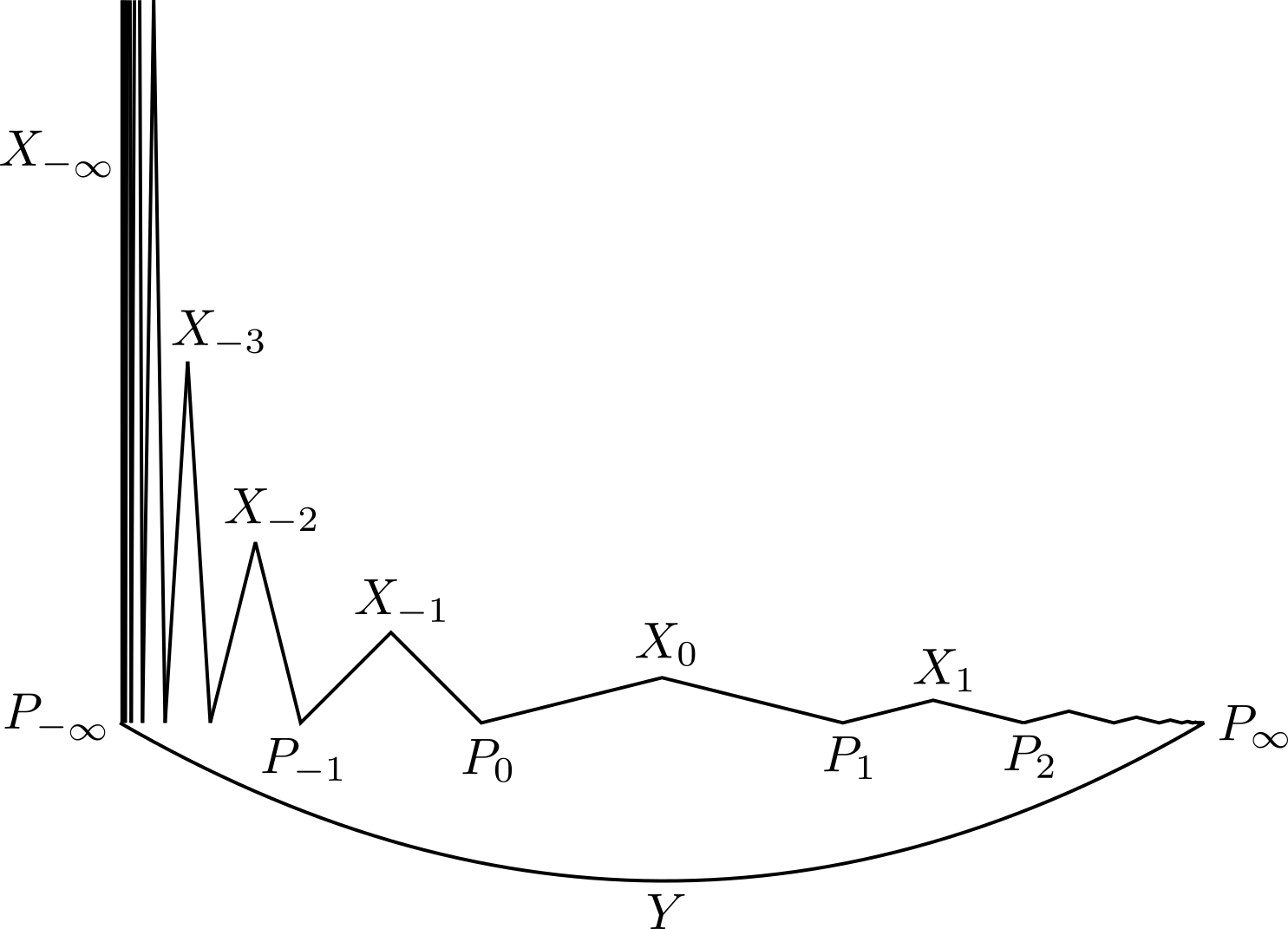}
		\caption{bounded subset of X}
	\end{figure}
	with the metric induced by $\mathbb{R}^2$.\\
	\emph{Step 2: The dynamics.}
	We characterize the dynamics by phases of length one. To each phase there corresponds a random variable
	$\xi_m$ where $( \xi_m)_{m\in \mathbb{Z}} $ is a sequence of 
	independent identically distributed random variables with
	$\mathbb{P} \left( \xi_0 = k \right) = 2^{-k}$ for $k \in \mathbb{N}$.
	In a phase with corresponding $\xi_m = k$ all points to the right of $P_{-(k+1)! +1}$ get pushed
	$k!$ triangles to the right 
	and all points on the lower half of the triangles to the left of $P_{-(k+1)!}$ decrease their height. \\
	We describe the dynamics during a phase by a function 
	$f: \left\{ 0 \leq s \leq t \leq1 \right\} \times \mathbb{N} \times X \mapsto X$. Let $f$ be such that
	\begin{itemize}
		\item
			$P \mapsto f_{0,t} (k,P)$ is bijective
		\item
			$f_{s,t} = f_{0,t} \circ f_{0,s}^{-1}$
		\item 
			$(s,t) \mapsto f_{s,t} (k,P) $ is continuous
		\item
			if $z \geq -(k+1)!$ and $P = (x,y) \in X_z^R $, then $f_{0,1} (k,P) \in 
			\left\{	(\tilde{x}, \tilde{y}) \in X_{z +k!}^R : \tilde{y} = 2^{-k!} y \right\}$
		\item
			if $z \geq -(k+1)!+1$ and $P = (x,y) \in X_z^L $, then $f_{0,1} (k,P) \in 
			\left\{	(\tilde{x}, \tilde{y}) \in X_{z +k!}^L : \tilde{y} = 2^{-k!} y \right\}$
		\item 
			if $z \geq -(k+1)!+1$ and $P \in X_{z-1}^R \cup X_z^L$, then
			$ \left| f_{0,t}(k,P) - f_{0,t} (k,P_z) \right| \leq \left| P - P_z \right|$
		\item
			if $z \leq -(k+1)!$ and $P=(x,y) \in X_{z}^L$ with $ y\leq 2^{-z-1}$, then 
			\linebreak $f_{0,t} (k,P) \in 
			\left\{	(\tilde{x}, \tilde{y}) \in X_{z}^L : \tilde{y} = 2^{-t}y \right\}$
		\item
			if $z \leq -(k+1)! -1$ and $P=(x,y) \in X_{z}^R$ with $ y\leq 2^{-z-1}$, then 
			\linebreak $f_{0,t} (k,P) \in 
			\left\{	(\tilde{x}, \tilde{y}) \in X_{z}^R : \tilde{y} = 2^{-t}y \right\}$
		\item
			if $P,Q \in X_z^L$ or $P,Q \in X_z^R$ for $z \in \mathbb{Z}$, then 
			$ \left| f_{s,t}(k,P) - f_{s,t} (k,Q) \right| \leq 4(k!+1) \left| P - Q \right|$
		\item
			if $P \in X_{-\infty}$ and $P=(-1,y)$, then $f_{0,t} (k, P) = (-1, 2^{-t}y)$
		\item
			if $P \in Y$, then $f_{0,t}(k,P)=P$.
	\end{itemize}
	Then, $t \mapsto f_{s,t}(\xi_m,P)$ describes the dynamics of the system 
	started in a point $P$ at time $s$  in a phase with corresponding random variable $\xi_m$. 
	Since  $(s,t) \mapsto f_{s,t}(k,P)$ is continuous and $P \mapsto f_{s,t}(k,P)$ is Lipschitz continuous with
	Lipschitz constant depending on $k$, the map
	$(s,t,P) \mapsto f_{s,t}(k,P)$ is continuous. \\
	In the following steps we show that the weak attractor of this system exists  and is not connected. \\
	\emph{Step 3: Attractor of discrete-time system.}
	Let $r \in \mathbb{N}$ be arbitrary. Define the bounded set 
	$K_r := \left\{ (x,y) \in X : y \leq 2^{r} \right\}$ and
	the neighborhood 
	$U_r = \left\{  (x,y) \in X:  y \leq 2^{-r} \right\}$ of $ \,\bigcup_{z \in \mathbb{Z}} P_z \cup Y$.
	Consider the discrete-time system generated by the iterated functions  
	$(f_{0,1}(\xi_m, \cdot))_{m \in \mathbb{Z}}$. If $\xi_m \geq k$ for some phase with $k!\geq 2r$,
	then the process started in $\bigcup_{z=-r}^{\infty} X_z \cap K_r$  
	stays in $U_r$ after this phase. Running $2r$ phases, all points in 
	$ K_r \cap \left( \bigcup_{i=r+1}^\infty X_{-i}  \cup X_{-\infty} \right)$ 
	decrease their height and reach $U_r$.
	Therefore, after $2r$ phases where at least one corresponding $\xi_m \geq k$ with $k!\geq 2r$
	the discrete-time process started in $K_r$ is in $U_r$. In contrast to the continuous-time process, the 
	discrete-time process cannot leave $U_r$ afterwards. By \cite[Theorem 3.4]{Crauel2001}, 
	there exists a pullback attractor of the discrete-time process and this 
	attractor is a subset of $ \bigcup_{z \in \mathbb{Z}} P_z \cup Y$.
	For $n \in \mathbb{N}$ define 
	\begin{align*}
		F_n (\xi_{-1}, \xi_{-2},  \dots, \xi_{-n}) := 
		f_{0,1} (\xi_{-1}, \cdot) \circ f_{0,1} (\xi_{-2}, \cdot) \circ \dots \circ 
		f_{0,1} (\xi_{-n}, \bigcup_{z \in \mathbb{Z}} P_z) \subset \bigcup_{z \in \mathbb{Z}} P_z .
	\end{align*}	 
	By definition of the pullback attractor, $F_n (\xi_{-1}, \xi_{-2},  \dots, \xi_{-n})$ converges to the
	pullback attractor as $n$ goes to infinity $\mathbb{P}$-almost surely. 
	Therefore, $P_0 \in F_n$ for large enough $n$ implies that $P_0$ is in the attractor as well.
	The point $P_0$ is not in $F_n$ iff there exist $k \in \mathbb{N}$ and 
	times $-n \leq t_0 < t_1 < \dots < t_k <0$ 
	such that $\xi_{t_i}=k$ for all $ 0 \leq i \leq k$ and $\xi_s \leq k$ for all $ t_0 \leq s <0$. 
	Then,
	\begin{align*}
		\mathbb{P}( P_0 \textrm{ is in the attractor}) &= 
		\lim_{n\rightarrow \infty} \mathbb{P} \left(P_0 \in F_n (\xi_{-1}, \xi_{-2},  \dots, \xi_{-n}) \right) \\
		&\geq 1 - \sum_{k \in \mathbb{N}} \mathbb{P} \left( \xi_0 = k | \xi_0 \geq k \right)^{k+1} 
		= \frac{1}{2}
	\end{align*}
	which implies that the pullback attractor is not connected with positive probability.
        More generally, the attractor is not connected if there exists an $m \geq 0$ such that 
	for all $n \in \mathbb{N}$ the point $P_0 \in F_n (\xi_{-m-1}, \xi_{-m-2},  \dots, \xi_{-m-n}) $. 
	This event is in the terminal sigma algebra.
	By Kolmogorov's zero-one law, the pullback attractor of the discrete-time system 
	is almost surely not connected.
	\\
	\emph{Step 4: Attractor of continuous-time system.}
	When we consider the continuous-time system we need to add a random phase shift which is uniformly distibuted on $[0,1)$. 
	For $0 \leq s,t <1$ and $n \in \mathbb{N}$, the system started in a point $P$ at time $s$ of a phase
	is described by
	\begin{align*}
		\varphi_{-s+n+t} ( \omega ,P) 
		= f_{0,t}( \xi_n, \cdot) \circ f_{0,1}( \xi_{n-1}, \cdot) \circ \dots \circ f_{0,1}( \xi_{1}, \cdot)
			\circ f_{s,1}( \xi_{0}, P).
	\end{align*}	
	with $\omega =  (s, (\xi_m)_{m \in \mathbb{Z}})  \in [ 0,1 ) \times \mathbb{N}^{\mathbb{Z}}=: \Omega$
	and canonical shift on $\Omega$ and the basic probability measure on $\Omega$ is the product of 
        Lebesgue measure on $[0,1)$ and the laws of   $(\xi_m)_{m \in \mathbb{Z}}$.  
	Then, $\varphi$ is a jointly continuous RDS as a composition of jointly continuous maps.
	\\
	Let $r \geq 2$. 
	If we start in a set $K_r$ as in Step $3$  in an incomplete phase with 
	corresponding $\xi_m \leq r$, then at the end of this phase the process is still in $K_{r}$. 
	The pullback attractor of the discrete-time system attracts this bounded set. 
	Hence, there exists a time $n_r \in \mathbb{N}$ 
	such that the discrete process started in $K_{r}$ 
	stays in a ball around the discrete-time attractor with radius $2^{-(r+1)!}$ 
	after time $n_r$ with probability $1-2^{-r}$. 
	\\
	We extend the discrete-time attractor to continuous time in such a way 
	that the so constructed random set stays strictly invariant under the given dynamics. 
	If one starts the end phase in a ball around the discrete-time attractor with radius $2^{-(r+1)!}$,
	one can leave the ball around the invariantly extended random set with radius $2^{-(k+1)!}$ 
	only during a phase with corresponding $\xi_m \geq r$. 
	\\
	Combining these three parts, the continuous-time process started in $K_r$ at time $t \geq n_r +1$
	is in a ball around the discrete-time attractor with radius $2^{-(r+1)!}$ with
	probability $1 - 2^{-r+1}$.
	\\
	This probability tends to one as $r$ goes to infinity.
	Therefore, the continuous-time extension of the discrete-time attractor is the weak attractor of
	the continuous-time system. By construction, the weak attractor of the continuous system is almost surely
	not connected. Note that the weak attractor will not almost surely be contained in the set $\bigcup_{z \in \mathbb{Z}} P_z \cup Y$.
\end{example}

\begin{remark}
	If every compact set in $X$ can be covered by a connected compact set, 
	then the weak attractor is connected. 
	This follows by the same arguments as in \cite[Proposition 3.7]{Crauel2001} 
	where this result was stated for the pullback attractor.
	Here, one does not need to assume continuity in time. \\
	This assumption is in particular satisfied for an attractor that attracts bounded sets in probability
	on a connected and locally connected Polish space.
	By local connectedness, a compact set can be covered by finitely many open connected sets. 
	Since a connected and locally connected Polish space is also path-connected 
	(see Mazurkiewicz-Moore-Menger theorem in \cite[p. 254, Theorem 1 and p. 253, Theorem 2]{Kuratowski1968}),
	one can connect these sets by paths.
\end{remark}

\bibliographystyle{plain}
\bibliography{mybib}

\end{document}